\documentclass{amsart}
\usepackage{latexsym, url}
\usepackage{hyperref}
\usepackage[utf8]{inputenc}
\usepackage{amsthm}
\usepackage{amsmath}
\usepackage{amsfonts}
\usepackage{amssymb}
\usepackage[final]{showkeys}    
\usepackage[dvips]{graphicx}
\usepackage{xypic}
\input xy
\xyoption{all}

\newtheorem{theo}{Theorem}[section]
\newtheorem{thm}[theo]{Theorem}

\newtheorem{lem}[theo]{Lemma}
\newtheorem{fact}[theo]{Fact}

\newtheorem{defin}[theo]{Definition}

\newtheorem{remark}[theo]{Remark}

\newtheorem{example}[theo]{Example}

\newtheorem{question}[theo]{Question}

\newcommand\Mod{\operatorname{Mod}}
\newcommand\Emb{\operatorname{Emb}}

\newcommand\op{\operatorname{op}}
\newcommand\id{\operatorname{id}}

\newcommand\Set{\operatorname{\bf Set}}

\newcommand\Str{\operatorname{\bf Str}}

\newcommand\ca{\mathcal {A}}

\newcommand\ci{\mathcal {I}}

\newcommand\ck{\mathcal {K}}
\newcommand\cl{\mathcal {L}}

\newcommand{\ran}{\operatorname{ran}}

\newcommand{\K}{\mathbf{K}}
\newcommand{\leap}[1]{\le_{#1}}
\newcommand{\lea}{\leap{\K}}
\newcommand{\ccl}{\mathbf{cl}}
\newcommand{\cclp}[1]{\ccl_{#1}}

\newcommand{\ba}{\bar{a}}
\newcommand{\bb}{\bar{b}}

\newcommand{\bx}{\bar{x}}
\newcommand{\by}{\bar{y}}

\newcommand{\fct}[2]{{}^{#1} #2}
\newcommand{\rest}{\upharpoonright}
\newcommand{\bigK}{\widehat{\K}}
\newcommand{\bigM}{\widehat{M}}

\newcommand{\bigtau}{\widehat{\tau}}
\newcommand{\LS}{\operatorname{LS}}
\newcommand{\Ll}{\mathbb{L}}

 \newbox\noforkbox \newdimen\forklinewidth
\forklinewidth=0.3pt \setbox0\hbox{$\textstyle\smile$}
\setbox1\hbox to \wd0{\hfil\vrule width \forklinewidth depth-2pt
 height 10pt \hfil}
\wd1=0 cm \setbox\noforkbox\hbox{\lower 2pt\box1\lower
2pt\box0\relax}

\newcommand{\seq}[1]{\langle #1 \rangle}

\title[Universal AECs and locally multipresentable categories]
      {Universal abstract elementary classes and locally multipresentable categories}

\author[M. Lieberman]{Michael Lieberman}
\email{lieberman@math.muni.cz}
\urladdr{http://www.math.muni.cz/\textasciitilde lieberman/}
\address{Department of Mathematics and Statistics, Faculty of Science, Masaryk University, Brno, Czech Republic}

\author[J. Rosick\'y]{Ji\v r\'i Rosick\'y}
\email{rosicky@math.muni.cz}
\urladdr{http://www.math.muni.cz/\textasciitilde rosicky/}
\address{Department of Mathematics and Statistics, Faculty of Science, Masaryk University, Brno, Czech Republic}
\thanks{The first and second authors are supported by the Grant Agency of the Czech Republic under the grant P201/12/G028.}

\author[S. Vasey]{Sebastien Vasey}
\email{sebv@math.harvard.edu}
\urladdr{http://math.harvard.edu/\textasciitilde sebv/}
\address{Department of Mathematics \\ Harvard University \\ Cambridge, Massachusetts, USA}

\date{\today
}

\begin{document}

\begin{abstract}
  We exhibit an equivalence between the model-theoretic framework of universal classes and the category-theoretic framework of locally multipresentable categories. We similarly give an equivalence between abstract elementary classes (AECs) admitting intersections and locally polypresentable categories. We use these results to shed light on Shelah's presentation theorem for AECs.
\end{abstract} 

\maketitle

\tableofcontents

\section{Introduction}

Abstract elementary classes (AECs) \cite{sh88, shelahaecbook} were introduced by Shelah in the late seventies as a purely semantic framework in which to study the model theory of infinitary logics such as $\Ll_{\infty, \omega}$. At about the same time, Lair \cite{lair-accessible} introduced the notion of an accessible category (which he called a ``catégorie modelable'') and proved that accessible categories are exactly those that can be sketched. Independently, the second author showed in his doctoral thesis \cite{rosicky-abstract,rosicky-thesis} that accessible categories correspond to classes of models of an $\Ll_{\infty, \infty}$-sentence, thus exhibiting an explicit connection with model theory. Later Makkai and Paré \cite{makkai-pare} independently rediscovered these results and showed that they say essentially the same thing in different languages. The connection between AECs and accessible categories was then studied more closely by the first author \cite{lieberman-categ} and independently by Beke and the second author \cite{beke-rosicky}. These works characterized AECs as special kinds of accessible categories with all directed colimits and whose morphisms are monomorphisms.

In recent joint work with Boney and Grossberg \cite{mu-aec-jpaa} the authors introduced $\mu$-AECs, a generalization of AECs where only closure under $\mu$-directed colimits is required. It was shown that $\mu$-AECs correspond exactly to accessible categories whose morphisms are monomorphisms (the correspondence is given in terms of an equivalence of categories: see Fact \ref{mu-aec-acc} for a more precise statement). This gives evidence that these frameworks are natural.

One of the questions that motivated Shelah's introduction of AECs is the eventual categoricity conjecture: an AEC categorical in a single high-enough cardinal (i.e.\ with a unique model of that cardinality, up to isomorphism) should be categorical in \emph{all} high-enough cardinals. This conjecture has profoundly shaped the development of the field (see the introduction of \cite{ap-universal-apal} for a survey and history of the conjecture). Note that Shelah's eventual categoricity conjecture can be rendered as a purely category-theoretic statement, with cardinalities replaced by presentability ranks---see \cite[\S6]{beke-rosicky})---and hence can be posed in relation to general accessible categories. 

Recently, the third author \cite{ap-universal-apal, categ-universal-2-selecta} proved that the conjecture holds in the more restricted framework of universal classes: classes of structures closed under isomorphism, substructures, and union of chains. Universal classes had previously been studied by, among others, Tarski \cite{tarski-th-models-i} and Shelah \cite{sh300-orig}.

In this paper, we show (Theorem \ref{char}) that universal classes have a natural category-theoretic analog: the locally $\aleph_0$-multipresentable categories (whose morphisms are monomorphisms) introduced by Diers \cite{diers}. They can be characterized as the $\aleph_0$-accessible categories that have all connected limits, i.e. those generated from equalizers and wide pullbacks (see Fact \ref{loc-limit-charact}). More generally, we prove that $\mu$-universal classes (where we allow the relations and functions to have arity strictly less than $\mu$) correspond to locally $\mu$-multipresentable categories.

After the initial submission of this paper, Hyttinen and Kangas \cite{hyttinen-kangas-universal-v2} have continued the work of the third author by proving that, in a technical model-theoretic sense, the big-enough models in an eventually categorical universal classes look like either vector spaces or sets. In light of the present paper, it would be interesting to know whether their result has a category-theoretic analog.

Another type of AEC that was studied in the third author's proof of the categoricity conjecture for universal classes are those which admit intersections. They were introduced by Baldwin and Shelah \cite[1.2]{non-locality} and can be characterized as the AECs admitting a certain closure operator. Any universal class admits intersections and the third author has shown \cite[5.27]{ap-universal-apal} that the categoricity conjecture holds in AECs admitting intersections if we assume a large cardinal axiom. We suggest that such AECs are also natural by characterizing them (Theorem \ref{admit-inter-charact}) as the locally $\aleph_0$-polypresentable categories of Lamarche \cite{lamarche-thesis}. These are the $\aleph_0$-accessible categories with wide pullbacks (see Fact \ref{loc-limit-charact}).  Indeed, this result generalizes as well: per Theorem~\ref{admit-inter-charact}, any $\mu$-AEC admitting intersections is a locally $\mu$-polypresentable category. Examples of locally multi- and polypresentable categories arise naturally in areas such as functional analysis, algebra, and even computer science, see Example \ref{categ-examples}. 

In the final section of this paper, we present a generalization of Shelah's presentation theorem \cite[I.1.9]{shelahaecbook} to all accessible categories with arbitrary $\mu$-directed colimits whose morphisms are monomorphisms: we show (Theorem \ref{pres-thm}) that any such category is the essential image by a faithful functor of a $\mu$-universal class. This generalizes the category-theoretic presentation theorem of the first and second authors \cite[2.5]{ct-accessible-jsl} as well as Boney's presentation theorem for metric AECs \cite[6.3]{boney-pres-metric-mlq}.  (Recall that metric AECs are an analogue of AECs in which the structures have underlying complete metric spaces, rather than sets: see \cite[2.7]{shelus} or \cite[2.1]{hirvhyt}.) We give one more characterization of the $\mu$-AECs admitting intersections as those classes for which the functor is, in a sense, close to being full (see Definition \ref{weakly-full-def} and Theorem \ref{pres-thm-char}).

Throughout this paper, we assume the reader is familiar with basic category theory as presented e.g.\ in \cite{joy-of-cats}. We also assume some familiarity with $\mu$-AECs and their relationship with accessible categories \cite{mu-aec-jpaa}.

We use the following notational conventions: we write $K$ for a class of $\tau$-structures and $\K$ (boldface) for a pair $(K, \lea)$, where $\lea$ is a partial order. We write $\ck$ (script) for a category. We will abuse notation and write $M \in \K$ instead of $M \in K$. For a structure $M$, we write $U M$ for its universe, and $|U M|$ for the cardinality of its universe. We write $M \subseteq N$ to mean that $M$ is a substructure of $N$. For $\alpha$ an ordinal, we let $\fct{<\alpha}{A}$ [resp. $\fct{\alpha}{A}$] denote the set of sequences of length less than [resp. exactly] $\alpha$ with elements from the set $A$. We will abuse this notation as well, writing $\fct{<\alpha}{M}$ in place of $\fct{<\alpha}{U M}$.

The authors would like to thank Will Boney for comments that helped improve the presentation of this paper.

\section{Universal classes or classes with intersections}

The following concept was introduced by Tarski for finite vocabularies \cite{tarski-th-models-i} and by Shelah for infinite (but still finitary) vocabularies \cite{sh300-orig}. In this paper, we study it for potentially infinitary vocabularies:

\begin{defin}
  Let $\mu$ be a regular cardinal. $K$ is a \emph{$\mu$-universal class} if:

  \begin{enumerate}
  \item $K$ is a class of structures in a fixed $\mu$-ary vocabulary $\tau = \tau (K)$.
  \item $K$ is closed under isomorphism, unions of $\mu$-directed systems of $\tau$-substructure inclusions, and $\tau$-substructures.
  \end{enumerate}

  When $\mu = \aleph_0$, we omit it and just say that $K$ is a universal class.
\end{defin}

We say that an $\Ll_{\infty, \infty}$-sentence is \emph{universal} if it has the form $\forall \bx  \phi$, where $\phi$ is quantifier-free. We say that a theory is \emph{universal} if it consists only of universal sentences. Tarski has shown \cite{tarski-th-models-i} that a universal class in a finite (and finitary) vocabulary is the class of models of a universal $\Ll_{\omega, \omega}$-theory. The proof generalizes without difficulties to infinitary logics:

\begin{fact}[Tarski's presentation theorem]\label{tarski-pres}
  Let $\mu$ be a regular cardinal and $K$ be a class of structures in some $\mu$-ary vocabulary $\tau$. The following are equivalent:

  \begin{enumerate}
  \item\label{tarski-1} There is a set $\Gamma$ of quantifier-free $\Ll_{\infty, \mu}$-types such that $K$ is the class of all $\tau$-structures omitting $\Gamma$.
  \item\label{tarski-2} $K$ is the class of models of a universal $\Ll_{\infty, \mu}$ theory.
  \item\label{tarski-3} $K$ is a $\mu$-universal class.
  \end{enumerate}
\end{fact}
\begin{proof}[Proof sketch]
  That (\ref{tarski-1}) implies (\ref{tarski-2}) and (\ref{tarski-2}) implies (\ref{tarski-3}) can be easily verified. To see that (\ref{tarski-3}) implies (\ref{tarski-1}), let $K_0$ be the class of $\tau$-structures that are generated by a set of size less than $\mu$ and are not contained in any member of $K$. Let $\Gamma$ be the set of types that code each structure in $K_0$, and use a $\mu$-directed system argument to see that $K$ is the set of $\tau$-structures omitting $\Gamma$.
\end{proof}

Recall from \cite[\S2]{mu-aec-jpaa} that a \emph{($\mu$-ary) abstract class} is a pair $\K = (K, \le)$ such that $K$ is a class of structures is a fixed $\mu$-ary vocabulary $\tau = \tau (\K)$, and $\le$ is a partial order on $K$ that respects isomorphisms and extends the $\tau$-substructure relation. We say that such a $\K$ is a \emph{$\mu$-universal class} if $K$ is a $\mu$-universal class and $\lea = \subseteq$.

In any abstract class $\K$, there is a natural notion of morphism: we say that $f: M \rightarrow N$ is a \emph{$\K$-embedding} if $f$ is an isomorphism from $f$ onto $f[M]$ and $f[M] \lea N$. We can see an abstract class and its $\K$-embeddings as a category. In fact (see \cite[\S2]{mu-aec-jpaa}), an abstract class is a replete and iso-full subcategory of the category of $\tau$-structures with injective homomorphisms.

We now recall the definition of a $\mu$-AEC from \cite[2.2]{mu-aec-jpaa}:

\begin{defin}
  Let $\mu$ be a regular cardinal. An abstract class $\K$ is a \emph{$\mu$-abstract elementary class} (or \emph{$\mu$-AEC} for short) if it satisfies the following three axioms:

  \begin{enumerate}
  \item Coherence: for any $M_0, M_1, M_2 \in \K$, if $M_0 \subseteq M_1 \lea M_2$ and $M_0 \lea M_2$, then $M_0 \lea M_1$.
  \item Chain axioms: if $\seq{M_i : i \in I}$ is a $\mu$-directed system in $\K$, then:
    \begin{enumerate}
    \item $M := \bigcup_{i \in I} M_i$ is in $\K$.
    \item $M_i \lea M$ for all $i \in I$.
    \item If $M_i \lea N$ for all $i \in I$, then $M \lea N$.
    \end{enumerate}
  \item Löwenheim-Skolem-Tarski (LST) axiom: there exists a cardinal $\lambda = \lambda^{<\mu} \ge |\tau (\K)| + \mu$ such that for any $M \in \K$ and any $A \subseteq U M$, there exists $M_0 \in \K$ with $M_0 \lea M$, $A \subseteq U M_0$, and $|U M_0| \le |A|^{<\mu} + \lambda$. We write $\LS (\K)$ for the least such $\lambda$.
  \end{enumerate}
\end{defin}

Note that when $\mu = \aleph_0$, we recover Shelah's definition of an AEC from \cite{sh88}. Note, too, that a $\mu$-universal class, equipped with the $\tau$-substructure relation, is a $\mu$-AEC. In this case, we may call it a \emph{universal $\mu$-AEC}. Equivalently, a universal $\mu$-AEC is a $\mu$-AEC $\K$ such that for any $N \in \K$, $M \subseteq N$ implies that $M \in \K$ and $M \lea N$.

More general than $\mu$-universal classes are $\mu$-AECs admitting intersections. AECs admitting intersections were introduced in \cite[1.2]{non-locality} and further studied in \cite[\S2]{ap-universal-apal}.

\begin{defin}
  A $\mu$-AEC $\K$ \emph{admits intersections} if for any $N \in \K$ and any $A \subseteq U N$, the set

  $$
  \cclp{\K}^N (A) = \ccl^{N} (A) := \bigcap \{M \in \K \mid M \lea N, A \subseteq U M\}
  $$
is the universe of a $\lea$-substructure of $M$. In this case, we abuse notation and write $\ccl^{N} (A)$ for this substructure as well.
\end{defin}

Note that any universal $\mu$-AEC $\K$ admits intersection, since $\ccl^N (A)$ is a substructure of $N$ and $\K$ is closed under substructures. On the other hand, the AEC of algebraically closed fields admits intersections but is not a universal class ($\mathbb{Q}$ is a subfield of an algebraically closed field, but it is not algebraically closed).

Many of the properties of AECs admitting intersections proven in \cite[\S2]{ap-universal-apal} can be generalized to $\mu$-AECs: it suffices to replace $\aleph_0$ by $\mu$ in the proofs. For example:

\begin{fact}\label{cl-props}
  Let $\K$ be a $\mu$-AEC that admits intersection.

  \begin{enumerate}
  \item If $M \lea N$ and $A \subseteq U M$, then $\ccl^{M} (A) = \ccl^{N} (A)$.
  \item Local character: If $B \subseteq \ccl^N (A)$ and $|B| < \mu$, then there exists $A_0 \subseteq A$ such that $|A_0| < \mu$ and $B \subseteq \ccl^N (A_0)$.
  \end{enumerate}
\end{fact}

We will use these facts without much comment in the sequel.

Similarly to \cite[3.1]{ap-universal-apal}, one can define the notion of a pseudo-universal $\mu$-AEC:

\begin{defin}
  A $\mu$-AEC $\K$ is \emph{pseudo-universal} if it admits intersections and whenever $f, g: M \rightarrow N$ are $\K$-embeddings and $A \subseteq U M$ is such that $f \rest A = g \rest A$, then $f \rest \ccl^M (A) = g \rest \ccl^M (A)$.
\end{defin}
\begin{remark}\label{univ-pseudo-univ}
  A universal $\mu$-AEC is pseudo-universal: since $\ccl^{M} (A)$ is just the closure of $A$ under the functions of $M$, any $\K$-embedding with domain $\ccl^M (A)$ is determined by its restriction to $A$.
\end{remark}

Some motivation may be in order: consider the class of groups in the language containing a symbol for multiplication, identity, and inverse. Then this is a universal class. On the other hand, if we look at the class of groups in the language containing only multiplication and identity, then it is not universal: substructures need not be closed under inverses. However it will still be pseudo-universal, since inverses are definable from the rest. In fact, we will see shortly (Theorem \ref{pseudo-univ-equiv}) that this is how all pseudo-universal classes look: any such class admits a functorial expansion to a universal one.

Although we will not need it in this paper, we also note that we can imitate the proof of \cite[3.7]{ap-universal-apal} to show that any pseudo-universal $\mu$-AEC is fully $(<\mu)$-tame and short over $\emptyset$.

We now show that pseudo-universal $\mu$-AECs are the same as universal $\mu$-AECs, up to extension of the vocabulary. Recall from \cite[3.1]{sv-infinitary-stability-afml} the notion of a functorial expansion of an abstract class $\K$: roughly, it is an abstract class $\bigK$ in an expansion of $\tau (\K)$ such that the reduct map is an isomorphism of concrete categories. We show that every pseudo-universal class (that does not contain the empty structure) can be functorially expanded to a universal class. This result is not needed in the rest of this paper, where we deal with equivalence of categories, but shows that every pseudo-universal class is \emph{isomorphic} (not just equivalent) to a universal class.

\begin{thm}\label{pseudo-univ-equiv}
  Let $\K$ be a pseudo-universal $\mu$-AEC that does not contain the empty structure. Then there is a functorial expansion $\bigK$ of $\K$ which is a universal $\mu$-AEC. Moreover, $|\tau (\bigK)| \le 2^{\LS (\K)}$.
\end{thm}
\begin{proof}
  Any abstract class admits a notion of type, called Galois (or orbital) types in the literature. They can be defined as the finest notion of type preserving $\K$-embeddings (see \cite[2.16]{sv-infinitary-stability-afml} for a formal definition). The proof, in short, is to take the $(<\mu)$-Galois Morleyization (i.e.\ expand the language with a symbol for every type of length less than $\mu$, see \cite[3.3]{sv-infinitary-stability-afml}) and then add Skolem functions. This works since types are able to code the closure operator. We now give a self-contained implementation of this proof.

  We will write $\ccl^M (\ba)$ instead of $\ccl^M (\ran{\ba})$ (here $\ran{\ba}$ denotes the range of $\ba$, i.e.\ the set of elements in the sequence $\ba$) and $\ba b$ for the concatenation of the sequence $\ba$ with the element $b$. For $M_1, M_2 \in \K$, $\ba_\ell \in \fct{<\mu}{M_\ell}$, $\ell = 1,2$, write $(\ba_1, M_1) \equiv (\ba_2, M_2)$ if there exists an isomorphism $f: \ccl^{M_1} (\ba_1) \cong \ccl^{M_2} (\ba_2)$ such that $f (\ba_1) = \ba_2$. Note that $\equiv$ is an equivalence relation with at most $2^{\LS (\K)}$-many classes. We say that an equivalence class $C$ \emph{codes closure} if $C = [(\ba b, M)]_{\equiv}$ implies that $b \in \ccl^M (\ba)$. Note that if $C$ codes closure, $C = [(\ba b, M)]_{\equiv}$, and $(\ba b', M) \equiv (\ba b, M)$, then $b' = b$. Indeed, by definition there is an isomorphism $f: \ccl^{M} (\ba b) \cong \ccl^{M} (\ba b')$ such that $f (\ba b) = f (\ba b')$. This is in particular an isomorphism $f: \ccl^{M} (\ba) \cong \ccl^{M} (\ba)$ sending $\ba$ to $\ba$. The identity is another such isomorphism, hence by pseudo-universality $f$ must be the identity.

  Fix a well-ordering $\preceq$ on the set of equivalence classes coding closure. Let $\bigtau$ consist of $\tau (\K)$ together with a new $\alpha$-ary function symbol $f_C$ for each $C$ coding closure with $C = [(\ba b, M)]_{\equiv}$, $\ell (\ba) = \alpha$. For each $M \in \K$, define a $\bigtau$-expansion $\bigM$ by setting $f_C^{\bigM} (\ba)$ to be the unique $b \in \ccl^{M} (\ba)$ such that $C = [(\ba b, M)]_{\equiv}$, provided it exists. Otherwise, let $D$ be $\preceq$-least coding closure such that there exists $b' \in M$ with $D = [(b', M)]_{\equiv}$ and set $f_C^{\bigM} (\ba) = b'$. Note that such a $D$ always exists as $\ccl^{M} (\emptyset) \neq \emptyset$ (we are assuming the empty structure is not part of $\K$).

  Let $\widehat{K} := \{\bigM \mid M \in \K\}$. It is easy to check that $\widehat{K}$ is a functorial expansion of $\K$ (see also \cite[3.5]{sv-infinitary-stability-afml}), hence we can let $\bigK := (\widehat{K}, \leap{\bigK})$, where $M \leap{\bigK} N$ if and only if $M \subseteq N$ and $M \rest \tau (\K) \lea N \rest \tau (\K)$. We claim that $\bigK$ is a universal $\mu$-AEC. It is enough to show that for any $N \in \bigK$, $M \subseteq N$ implies that $M \leap{\bigK} N$. Indeed, noting that $\K$, and hence $\bigK$, admits intersections, it suffices to show that $\ccl^{N} (M) = M$. In what follows, we do not distinguish between $\equiv$ computed in $\K$ and $\bigK$ and between the closure operations $\ccl_{\K}$ and $\ccl_{\bigK}$: since $\bigK$ is a functorial expansion of $\K$, they are the same up to expansion of the language.

  Clearly, $M \subseteq \ccl^N (M)$, so let us check that $\ccl^N (M) \subseteq M$. Let $b \in \ccl^N (M)$. By local character (Fact \ref{cl-props}), there exists $A \subseteq U M$ of size less than $\mu$ such that $b \in \ccl^N (A)$. Let $\ba$ be an enumeration of $A$. Let $C := [(\ba b, N)]_{\equiv}$. Then $C$ codes closure and $b \in N$, hence we must have that $f_C^N (\ba) = b$. Since $M \subseteq N$, $b = f_C^N (\ba) = f_C^M (\ba)$, so $b \in U M$, as desired.
\end{proof}

\begin{remark}\label{pseudo-univ-equiv-rmk}
  The restriction that $\K$ does not contain the empty structure is essential, as otherwise one would not be able to add constant symbols to the vocabulary to code the closure of the empty set. This is, nonetheless, not a serious restriction: given any $\mu$-AEC $\K$ there is a $\mu$-AEC $\K'$ that is isomorphic to $\K$ as a category and does not contain the empty structure: for each $M \in \K$, expand $M$ to a $(\tau (\K) \cup \{c\})$-structure $M' \in \K'$, where $c$ is a new constant symbol, such that $c^{M'} \notin U M$, $U M' = \{c^{M'}\} \cup U M$, and a relation $R^{M'} (\ba)$ holds if and only if either $c^{M'} \in \text{ran} (\ba)$ or $c^{M'} \notin \text{ran} (\ba)$ and $R^{M} (\ba)$.  For any function symbol $f\in\tau (\K)$, define $f^{M'} (\ba) = c^{M'}$ if $c^{M'} \in \text{ran} (\ba)$ and $f^{M'} (\ba) = f^{M} (\ba)$ otherwise. Finally, for $M_1', M_2' \in \K'$, let $M_1 \leap{\K'} M_2$ if and only if $M_1 \lea M_2$ and $c^{M_1'} = c^{M_2'}$.
  \end{remark}

\section{Accessible categories}

To present category-theoretic equivalents of $\mu$-universal classes and $\mu$-AECs admitting intersections, we require the notion of an accessible category (see \cite{makkai-pare} or \cite{adamek-rosicky}).

\begin{defin}
  Let $\ck$ be a category and let $\lambda$ be a regular cardinal.

  \begin{enumerate}
  \item An object $M$ is \emph{$\lambda$-presentable} if its hom-functor $\ck(M,-):\ck\to\Set$ preserves $\lambda$-directed colimits. Put another way, $M$ is $\lambda$-presentable if for any morphism $f:M\to N$ with $N$ a $\lambda$-directed colimit $\langle \phi_\alpha:N_\alpha\to N\rangle$, $f$ factors essentially uniquely through one of the $N_\alpha$, i.e.\ $f=\phi_\alpha f_\alpha$ for some $f_\alpha:M\to N_\alpha$.
  \item $\ck$ is \emph{$\lambda$-accessible} if it has $\lambda$-directed colimits and $\ck$ contains a set $S$ of $\lambda$-presentable objects such that every object of $\ck$ is a $\lambda$-directed colimit of objects in $S$.
  \item $\ck$ is \emph{accessible} if it is $\lambda'$-accessible for some regular cardinal $\lambda'$.
  \end{enumerate}
\end{defin}

Intuitively, an accessible category is a category with all  sufficiently directed colimits and such that every object can be written as a highly directed colimit of ``small'' objects.  Here ``small'' is interpreted in terms of {\em presentability}, a notion of size that makes sense in any (possibly non-concrete) category. In the category of sets, of course, a set is $\lambda$-presentable if and only if its cardinality is less than $\lambda$; in an AEC $\K$, the same is true for all $\lambda>\LS(\K)$.  

From \cite[\S4]{mu-aec-jpaa}, we have that $\mu$-AECs are the same as accessible categories whose morphisms are monomorphisms, up to equivalence of categories:

\begin{fact}\label{mu-aec-acc}
  If $\K$ is a $\mu$-AEC, then it is an $\LS (\K)^+$-accessible category with all $\mu$-directed colimits whose morphisms are monomorphisms. Conversely, any $\mu$-accessible category whose morphisms are monomorphisms is equivalent to a $\mu$-AEC.
\end{fact}

In general, a $\mu$-AEC need not be a $\mu$-accessible category and, in fact, the least cardinal $\lambda$ such that $\K$ is $\lambda$-accessible cannot be bounded by a function of $\mu$: this is the case, in short, because the index of accessibility is determined as much by the parameter $\LS(\K)$ as by $\mu$. For example, given a regular cardinal $\lambda$, the ($\aleph_0$-)AEC $\K$ of sets of cardinality at least $\lambda$ (ordered by subset inclusion) contains no $\lambda$-presentable objects, hence cannot be $\aleph_0$-accessible or even $\lambda$-accessible.  It is, however, $\lambda^+$-accessible, as $\LS(\K)=\lambda$. We can show, though, that $\mu$-AECs admitting intersections \emph{are} $\mu$-accessible:

\begin{lem}\label{inter-acc-lem}
  If $\K$ is a $\mu$-AEC admitting intersections, then $\K$ is $\mu$-accessible.
\end{lem}
\begin{proof}
  By definition of a $\mu$-AEC, $\K$ has all $\mu$-directed colimits. It remains to exhibit a set of $\mu$-presentable objects in $\K$ such that every element of $\K$ is a $\mu$-directed colimit thereof.

  Let us call $M \in \K$ \emph{$\mu$-generated} if there exists a set $A \subseteq U M$ with $|A| < \mu$ such that $M = \ccl^M (A)$. Clearly, there is only a set of isomorphism types of $\mu$-generated objects. Moreover by the local character property of the closure operator (Fact \ref{cl-props}), it follows that any $N \in \K$ can be written as $\bigcup_{A \subseteq U N, |A| < \mu} {\ccl^N (A)}$. In other words, every $N$ is a $\mu$-directed colimit of $\mu$-generated objects. Moreover it is easy to check that $\mu$-generated objects are $\mu$-presentable. This completes the proof.
\end{proof}

\section{Axiomatizability of accessible categories}

For completeness, we recall that accessible categories can be presented syntactically. For an $\Ll_{\infty, \infty}$-sentence $\phi$, we denote by $\Mod (\phi)$ the category of models of $\phi$ with homomorphisms (i.e.\ mappings preserving all functions and relations). We say that a category $\ck$ is \emph{axiomatizable} by $\phi$ if $\ck$ is equivalent to the category $\Mod (\phi)$. We call $\phi$ \emph{basic} if it is a conjunction of sentences of the form $\forall \bx (\phi_1 \rightarrow \phi_2)$, where $\phi_1$ and $\phi_2$ are positive existential formulas. Following \cite[p.~58]{makkai-pare}, let us call a category \emph{$(\infty, \mu)$-elementary} if it is axiomatizable by a \emph{basic} $\Ll_{\infty, \mu}$-sentence (this is the same as the category of models of an arbitrary $\Ll_{\infty, \mu}$-sentence ordered by elementarity according to a suitable fragment, see \cite[3.2.8]{makkai-pare}). 

A theorem of the second author \cite{rosicky-abstract, rosicky-thesis} asserts that accessible categories are precisely the $(\infty, \infty)$-elementary ones. More precisely, by \cite[3.2.3, 3.3.5, 4.3.2]{makkai-pare} we have:

\begin{fact}\label{lair-precise} \
  \begin{enumerate}
  \item Any $(\infty, \mu)$-elementary category is accessible and has all $\mu$-directed colimits.
  \item Any $\mu$-accessible category is $(\infty, \mu)$-elementary.
  \end{enumerate}
\end{fact}

In this paper, we are interested in abstract classes where the morphisms are substructure embeddings. For $\phi$ an $\Ll_{\infty, \infty}$-sentence, denote by $\Emb (\phi)$ the category of models of $T$ with substructure embeddings (i.e.\ \emph{injective} mappings preserving all functions and relations \emph{and} reflecting all relations). Thus the difference between $\Mod$ and $\Emb$ when all morphisms are mono is the same as the difference between graphs ordered by the subgraph relation and graphs ordered by the \emph{induced subgraph} relation. From the point of view of category theory, this is not a serious difference:

\begin{lem}\label{mod-emb}
  Let $\mu$ be an infinite cardinals and let $\ck$ be a category. The following are equivalent:

  \begin{enumerate}
  \item\label{mod-emb-1} $\ck$ is equivalent to $\Mod (\phi)$, for $\phi$ a basic $\Ll_{\infty, \mu}$-sentence, and all the morphisms of $\ck$ are monomorphisms.
  \item\label{mod-emb-2} $\ck$ is equivalent to $\Emb (\phi)$, for $\phi$ a basic $\Ll_{\infty, \mu}$-sentence.
  \end{enumerate}
\end{lem}
\begin{proof}[Proof idea]
  Assume (\ref{mod-emb-1}). Add a sort for each relation symbol and code relations as sets of tuples. Let $\phi'$ describe this. Then $\Mod (\phi)$ and $\Emb (\phi')$ are isomorphic. Therefore (\ref{mod-emb-2}) holds. Conversely, assume (\ref{mod-emb-2}). For each relation symbol $R$, add a relation symbol $\neg R$ coding the negation of $R$. Also add a relation symbol $\neq$ coding non-equality. Once again, we obtain a sentence $\phi'$ such that $\Mod (\phi')$ is isomorphic to $\Emb (\phi)$. Therefore (\ref{mod-emb-1}) holds.
\end{proof}

It follows, for example, (see Lemma \ref{inter-acc-lem} and Fact \ref{lair-precise}) that $\mu$-AECs admitting intersections are equivalent to the category of models of an $\Ll_{\infty, \mu}$-sentence. Note that the stronger statement that every $\mu$-AEC admitting intersections is \emph{directly} axiomatizable (without changing the category) by an $\Ll_{\infty, \mu}$-sentence is \emph{false}. Indeed, the AEC whose models are equivalence relations all of whose classes are countably infinite, ordered by the relation ``equivalence classes do not grow,'' is not the class of models of an $\Ll_{\infty, \omega}$-sentence.

Similarly, it is not known whether any AEC is directly $\Ll_{\infty, \LS (\K)^+}$-axiomatizable, but by Fact \ref{lair-precise}, it is \emph{equivalent} to a category of models of an $\Ll_{\infty, \LS (\K)^+}$-sentence.

It is also not known whether there is a natural category-theoretic definition of the $(\infty, \mu)$-elementary categories (see the discussion after \cite[4.3.3]{makkai-pare}). In fact, it is not known whether all $\mu$-AECs are $(\infty, \mu)$-elementary (this question was posed in \cite[\S5]{beke-rosicky} in relation to accessible categories with all morphisms monomorphisms).

\section{Locally multi- and polypresentable categories}

The concept of a locally presentable category is originally due to Gabriel and Ulmer \cite{gabriel-ulmer}. We will use the following definition: 

\begin{defin}
  For $\mu$ a regular cardinal, we say that a category $\ck$ is \emph{locally $\mu$-presentable} if it is $\mu$-accessible and has \emph{all} (small) colimits. We say it is \emph{locally presentable} if it is $\mu'$-locally presentable for some $\mu'$.
\end{defin}

Examples include the category of groups with group homomorphisms or any Grothendieck topos. Although we will move quickly to other characterizations, the initial definitions of, and motivations for, locally multipresentable and locally polypresentable categories arise from mathematically natural weakenings of the notion of colimit.  Just as colimits go back to initial objects, one can define a notion of \emph{multicolimit} that goes back to a \emph{multiinitial object}. This is a set $\ci$ of objects of a category $\ck$ such that for every object $M$ of $\ck$ there is a unique $i \in \ci$ and a unique morphism $f_i : i \rightarrow M$. For example, the category of fields and field homormorphisms does not have an initial object, but does have a multiinitial object: the prime fields in each characteristic. 

There is also the still weaker notion of a \emph{polyinitial} object. This is a set $\ci$ of objects of a category $\ck$ such that for every object $M$ in $\ck$:

\begin{enumerate}
\item There is a unique $i \in\ci$ having a morphism $i \to M$.
\item For each $i \in \ci$, given $f,g:i \to M$, there is a unique (isomorphism) $h:i \to i$ with $fh=g$.
\end{enumerate}

For example, the algebraic closures of the prime fields form a polyinitial object in the category of algebraically closed fields.

The multicolimit of a diagram $D$ in a category $\ck$ is a multiinitial object in the category of cones on $D$, i.e. a set of cones such that for any cone on $D$, there will be a unique induced map from exactly one of the members of the set.  The polycolimit of a diagram is defined similarly.

One then defines locally multi- and polypresentable categories by replacing colimits by multicolimits and polycolimits, respectively:

\begin{defin}
  For $\mu$ a regular cardinal, we say that a category $\ck$ is \emph{locally $\mu$-[multi/poly]presentable} if it is $\mu$-accessible and has all [multi/poly]colimits. \emph{Locally [multi/poly]presentable} means locally $\mu'$-[multi/poly]presentable for some $\mu'$.  
\end{defin}
\begin{remark}
  The set $\ci$ in the definition of a multiinitial object is allowed to be empty. Thus in a locally $\mu$-multipresentable category, a diagram of the form

  $$  
  \xymatrix@=3pc{
    M_1 & \\
    M_0 \ar [u] \ar [r] &
    M_2 & \\
  }
  $$
may not even be completable: there might not exist morphisms  $M_\ell \rightarrow N$ making the diagram commute. In model-theoretic terms, the category may not have the amalgamation property (see \cite{coloring-classes-jsl} for examples of universal classes failing the amalgamation property in non-trivial ways).
\end{remark}
\begin{example}\label{categ-examples} \
  \begin{enumerate}
  \item As noted above, the category of all fields with field homomorphisms is locally $\aleph_0$-multipresentable.  It is \emph{not} locally $\aleph_0$-presentable as it cannot have an initial object---a consideration of characteristics makes clear that no single field can map into all of the others.  There is, however, an initial object relative to the fields of each fixed characteristic, which together form the multiinitial family (\cite[4.25(2), 4.29]{adamek-rosicky}).
  \item The category of linearly ordered sets with order-preserving maps is locally $\aleph_0$-multipresentable, with multiinitial family consisting of two objects, one for the connected linear orders and one for the disconnected (\cite[4.25(1), 4.29]{adamek-rosicky}).
  \item The category of pre-Hilbert spaces (i.e.\ we do not require completeness) with linear orthogonal maps is locally $\aleph_0$-multipresentable, while the category of Hilbert spaces with linear orthogonal maps is locally $\aleph_1$-multipresentable (\cite{diers}).
  \item As noted above, the category of algebraically closed fields with fields homomorphisms is locally $\aleph_0$-polypresentable (\cite{hebert}).  Here the polyinitial family consists of the closures of the prime fields.
  \item Lamarche's categories of aggregates \cite{lamarche-thesis} and Coquand's categories of embeddings \cite{coquand-embeddings} are examples of, respectively, locally $\aleph_0$-multipresentable and locally $\aleph_0$-polypresentable categories whose morphisms are monomorphisms.  Both notions arise in theoretical computer science, as efforts to model the phenomenon of polymorphism in type theory.
  \end{enumerate}
\end{example}

We can also characterize each of these categories in terms of limits:

\begin{fact}\label{loc-limit-charact}
  Let $\ck$ be a $\mu$-accessible category.

  \begin{enumerate}
  \item\cite{lamarche-thesis} $\ck$ is locally $\mu$-polypresentable if and only if $\ck$ has wide pullbacks (that is, limits over diagrams $A_i\to A$, $i\in I$, where $I$ is a set).
  \item \cite{diers}, \cite[4.30]{adamek-rosicky} $\ck$ is locally $\mu$-multipresentable if and only if $\ck$ has all connected limits.
  \item \cite{gabriel-ulmer}, \cite[2.47]{adamek-rosicky} $\ck$ is locally $\mu$-presentable if and only if $\ck$ has all limits.
  \end{enumerate}
\end{fact}

Note that the locally presentable categories whose morphisms are monomorphisms are not very interesting: since they have coequalizers, it must be the case that any two morphisms $f, g : M \rightarrow N$ are equal! In other words, between any two objects there is at most one morphism. Thus locally presentable categories whose morphisms are monomorphisms are exactly the complete lattices. In particular, they are small and their objects are rigid.

However, the situation is different for locally multi- and polypresentable categories:

\begin{lem}\label{mu-aec-categ}
  Let $\K$ be a $\mu$-AEC. If $\K$ admits intersections, then $\K$ is locally $\mu$-polypresentable. If in addition $\K$ is pseudo-universal, then $\K$ is locally $\mu$-multipresentable.
\end{lem}
\begin{proof} 
  By Lemma \ref{inter-acc-lem}, $\K$ is $\mu$-accessible. It is easy to check that having wide pullbacks in a $\mu$-AEC is the same as admitting intersections. Therefore $\K$ has wide pullbacks and hence by Fact \ref{loc-limit-charact} is locally $\mu$-polypresentable.

  Assume now that $\K$ is pseudo-universal. We check that $\K$ has connected limits, which is enough by Fact \ref{loc-limit-charact}. To see that $\K$ has connected limits, it suffices to check that $\K$ has wide pullbacks and equalizers (this is similar to the proof that having arbitrary limits is the same as having products and equalizers \cite[12.3]{joy-of-cats}). We have already checked that $\K$ has wide pullbacks. As for equalizers this is implied by the definition of a pseudo-universal class: assume we have two $\K$-embeddings $f, g: M \rightarrow N$. Let $A := \{x \in U M \mid f (x) = g (x)\}$. By definition of $A$, $f \rest A = g \rest A$. By the definining condition of pseudo-universality, $f \rest \ccl^M (A) = g \rest \ccl^M (A)$. Therefore $A = \ccl^M (A)$, so $A \lea M$. Therefore $h := f \rest A$ is a $\K$-embedding and hence the equalizer of $f$ and $g$.
\end{proof}

We obtain the following characterization of $\mu$-AECs admitting intersections:

\begin{thm}\label{admit-inter-charact}
  Let $\ck$ be a category and let $\mu$ be a regular cardinal. The following are equivalent:

  \begin{enumerate}
  \item $\ck$ is locally $\mu$-polypresentable and all its morphisms are monomorphisms.
  \item $\ck$ is equivalent to a $\mu$-AEC which admits intersections.
  \end{enumerate}
\end{thm}
\begin{proof}
  If $\ck$ is equivalent to a $\mu$-AEC which admits intersections, then by Lemma \ref{mu-aec-categ}, $\ck$ is locally $\mu$-polypresentable. Conversely, assume that $\ck$ is locally $\mu$-polypresentable and all its morphisms are mono. By Fact \ref{mu-aec-acc}, $\ck$ is equivalent to a $\mu$-AEC $\K$. Since $\ck$ has wide pullbacks, $\K$ has wide pullbacks, hence $\K$ must admit intersections.
\end{proof}

For locally $\mu$-multipresentable categories, we give a syntactic description. Johnstone \cite{johnstone-multipres} has shown (for $\mu = \aleph_0$) that they can be axiomatized by disjunctive theories (see the proof below), and we can then Skolemize to obtain a universal class.

\begin{lem}\label{char-lem} If $\ck$ is a locally $\mu$-multipresentable category with all morphisms monomorphisms, then $\ck$ is axiomatizable by a universal $\Ll_{\infty, \mu}$-theory.
\end{lem}
\begin{proof}
Consider the canonical embedding $E:\ck\to\Set^{\ca^{\op}}$ where $\ca$ is the representative full subcategory of $\mu$-presentable objects in $\ck$. Then $E$ preserves $\mu$-directed colimits and $\mu$-presentable objects. Following \cite[4.30]{adamek-rosicky}, $\ck$ is equivalent to a $\mu$-cone-orthogonality class in $\Set^{\ca^{\op}}$ and the latter category can be viewed as an equational variety of many-sorted universal algebras. Thus $\ck$ can be axiomatized by a disjunctive theory in the corresponding unary vocabulary $\tau$ (see \cite[Exercise 5.f]{adamek-rosicky}). Recall that this theory consists of sentences
$$
(\forall \bx)(\varphi(\bx)\to(\exists !\by) \bigvee\limits_{i\in I}\psi_i(\bx,\by))\wedge(\forall \bx,\by)\neg\bigvee\limits_{i\neq j\in I}(\psi_i(\bx,\by)\wedge\psi_j(\bx,\by)
$$
where $\varphi$ and $\psi_i$ are conjunctions of atomic formulas and $\bx,\by$ are strings of variables. The use of $\exists !$ can be eliminated by introducing a new relation symbol $R (\bx,\by)$ such that $R(\ba,\bb)$ if and only if 
$\bigvee\limits_{i\in I}\psi_i(\ba,\bb)$ and then adding Skolem functions. Thus $\ck$ can be axiomatized by a universal $\Ll_{\infty, \mu}$-theory.
\end{proof}

We obtain the following characterization of universal classes:

\begin{thm}\label{char}
  Let $\ck$ be a category and let $\mu$ be a regular cardinal. The following are equivalent:

  \begin{enumerate}
  \item\label{char-1} $\ck$ is locally $\mu$-multipresentable and all its morphisms are monomorphisms.
  \item\label{char-2} $\ck$ is equivalent to $\Emb (T)$, for some universal $\Ll_{\infty, \mu}$-theory $T$.
  \item\label{char-3} $\ck$ is equivalent to a universal $\mu$-AEC.
  \item\label{char-4} $\ck$ is equivalent to a pseudo-universal $\mu$-AEC.
  \end{enumerate}
\end{thm}
\begin{proof}
  (\ref{char-1}) implies (\ref{char-2}) is Lemma \ref{char-lem} combined with Lemma \ref{mod-emb}. (\ref{char-2}) implies (\ref{char-3}) is Fact \ref{tarski-pres}. (\ref{char-3}) implies (\ref{char-4}) is Remark \ref{univ-pseudo-univ}. Finally, (\ref{char-4}) implies (\ref{char-1}) is Lemma \ref{mu-aec-categ}.
\end{proof}

Note that a syntactic characterization of locally polypresentable classes (in terms of models of a pullback theory) is also known \cite{hebert} (see also Rabin's characterization of first-order theories with intersections \cite{rabin-inter}).

\subsection{Summary}
We have the following hierarchy of categories whose morphisms are monomorphisms, for $\mu$ a fixed regular cardinal:

\begin{enumerate}
\item Locally $\mu$-multipresentable.
\item Locally $\mu$-polypresentable.
\item Accessible.
\end{enumerate}

Each level is properly contained in the next and each level admits a characterization in terms of abstract classes (see Theorem \ref{char}, Theorem \ref{admit-inter-charact}, and Fact \ref{mu-aec-acc}):

\begin{enumerate}
\item Universal $\mu$-AEC.
\item $\mu$-AEC admitting intersections.
\item $\mu'$-AEC, for some $\mu'$.
\end{enumerate}

Each level also has a known syntactic characterization (see Theorem \ref{char}, \cite{hebert}, and Fact \ref{lair-precise}):

\begin{enumerate}
\item $\Emb (T)$, for $T$ a universal $\Ll_{\infty, \mu}$-theory.
\item $\Emb (T)$, for $T$ a pullback $\Ll_{\infty, \mu}$-theory.
\item $\Emb (T)$, for $T$ a basic $\Ll_{\infty, \infty}$-theory.
\end{enumerate}

We do not know where $\mu$-AECs (for fixed $\mu$) stand in this hierarchy:

\begin{question}\label{acc-mu-aec-q}
  Is every accessible category with $\mu$-directed colimits and all morphisms monomorphisms equivalent to a $\mu$-AEC?
\end{question}

\section{More on Shelah's presentation theorem}

Shelah's presentation theorem \cite[I.1.9]{shelahaecbook} says that any AEC is the reduct of a class of models of an $\Ll_{\omega, \omega}$ theory omitting a set of types. Moreover the reduct map is functorial. The proof generalizes to $\mu$-AECs \cite[3.2]{mu-aec-jpaa}. In \cite[2.5]{ct-accessible-jsl}, the following category-theoretic analog was proven (for $\mu = \aleph_0$, but the proof generalizes).

\begin{defin}
  A functor is called \emph{essentially surjective} if every object of its codomain is isomorphic to an object in its range.
\end{defin}

\begin{fact}\label{lr-pres}
  Let $\ck$ be an accessible category with $\mu$-directed colimits and all morphisms monomorphisms. Then there is a $\mu$-accessible category $\cl$ whose morphisms are monomorphisms and an essentially surjective faithful functor $F: \cl \rightarrow \ck$ preserving $\mu$-directed colimits.
\end{fact}

Fact \ref{lr-pres} says that accessible categories with $\mu$-directed colimits and all morphisms mono are, while not necessarily $\mu$-accessible, functorial images of $\mu$-accessible categories. This \emph{cannot} be obtained directly from Shelah's presentation theorem, since we do not know whether accessible categories with $\mu$-directed colimits are $\mu$-AECs (the problem is that we are fixing $\mu$, see Question \ref{acc-mu-aec-q}). 

In this section, we combine Fact \ref{lr-pres} with the proof of Shelah's presentation theorem \cite[I.1.9]{shelahaecbook} to obtain the following common generalization:

\begin{thm}\label{pres-thm}
  Let $\ck$ be an accessible category with $\mu$-directed colimits whose morphisms are monomorphisms. Then there is a universal $\mu$-AEC $\cl$ and an essentially surjective faithful functor $F: \cl \rightarrow \ck$ preserving $\mu$-directed colimits.
\end{thm}

Note that Theorem \ref{pres-thm} is indeed a generalization of Fact \ref{lr-pres}, since universal $\mu$-AECs are $\mu$-accessible (Lemma \ref{inter-acc-lem}). Moreover, Theorem \ref{pres-thm} is a generalization of Shelah's presentation theorem: by Tarski's presentation theorem (Fact \ref{tarski-pres}) every universal $\mu$-AEC is a class of models omitting a set of quantifier-free $\Ll_{\infty, \mu}$-types. Theorem \ref{pres-thm} also generalizes Boney's presentation theorem for metric AECs \cite[6.3]{boney-pres-metric-mlq}.

In Shelah's presentation theorem the functor $F$ from Theorem \ref{pres-thm} is the reduct map, as we now establish by imitating Shelah's proof:

\begin{lem}\label{pres-univ}
  Let $\K$ be a $\mu$-AEC which does not contain the empty structure. There exists an expansion $\tau'$ of $\tau := \tau (\K)$ and a universal $\mu$-AEC $\K'$ in the vocabulary $\tau'$ such that the reduct map $\K' \rightarrow \K$ is a faithful functor preserving $\mu$-directed colimits which is surjective on objects. Moreover, $|\tau'| = |\tau| + \LS (\K)$.
\end{lem}
\begin{proof}
  Let $\tau' := \tau \cup \{f_i^\alpha : i < \LS (\K), \alpha < \mu\}$, where each $f_i^\alpha$ is a new $\alpha$-ary function symbol. Let:

  $$
  \K' := \{M' \in \Str (\tau') \mid M' \rest \tau \in \K \land \forall A \subseteq U M' . \ccl^{M'} (A) \rest \tau \lea M' \rest \tau\}
  $$

  Here, $\Str (\tau')$ is the class of $\tau'$-structure and $\ccl^{M'} (A)$ denotes the closure of $A$ under the functions of $M'$. It is easy to check that $\K'$ is a $\mu$-universal class and that the reduct map is a faithful functor from $\K'$ to $\K$ preserving $\mu$-directed colimits. We show that it is onto: let $M \in \K$, and pick a $\mu$-directed system $\{M_{s} : s \in [M]^{<\mu}\}$ (where $[M]^{<\mu}$ denotes the set of subsets of $U M$ of cardinality strictly less than $\mu$) such that $|U M_s| \le \LS (\K)$, $s \subseteq U M_s$, and $s \subseteq t$ implies that $M_s \lea M_t$. For each $s \in [M]^{<\mu}$, let $\{c_i^{s} : i < \LS (\K)\}$ be an enumeration (possibly with repetitions) of $U M_{s}$. Finally, for $\ba \in \fct{\alpha}{M}$, $\alpha < \mu$, and $i < \LS (\K)$, define $\left(f_i^{\alpha}\right)^{M'} (\ba) := c_i^{\ran (\ba)}$. This works: check that for any $A \subseteq U M'$, $\ccl^{M'} (A) = \bigcup_{s \in [\ccl^{M'} (A)]^{<\mu}} M_s$ and use the chain axioms of $\mu$-AECs.
\end{proof}

\begin{proof}[Proof of Theorem \ref{pres-thm}]
  By Fact \ref{lr-pres}, there is a $\mu$-accessible category $\ck_0$ (whose morphisms are monomorphisms) and an essentially surjective faithful functor $F_0 : \ck_0 \rightarrow \ck$ preserving $\mu$-directed colimits. By Fact \ref{mu-aec-acc}, $\ck_0$ is equivalent via a functor $G$ to a $\mu$-AEC $\K_0$ ($G$ may not be surjective but it will be essentially surjective) and without loss of generality (see Remark \ref{pseudo-univ-equiv-rmk}) $\K_0$ does not contain the empty structure. By Lemma \ref{pres-univ}, there is a universal $\mu$-AEC $\cl$ such that the reduct functor $F_1 : \cl \rightarrow \K_0$ is a faithful functor preserving $\mu$-directed colimits. Let $F := F_1 G F_0$.
\end{proof}

If we ask for the functor $F$ from Theorem \ref{pres-thm} to be full, then it will be an equivalence of categories, so $\ck$ will be equivalent to a $\mu$-universal class. On the other hand, the reduct functor $F$ from Lemma \ref{pres-univ} is always surjective on morphisms (that is, if $f: A \rightarrow B$, then there exists $\bar{f}: A' \rightarrow B'$ such that $F (\bar{f}) = f$). Consider now the following intermediate weakening of fullness:

\begin{defin}\label{weakly-full-def}
  A functor $F: \cl \rightarrow \ck$ is \emph{pullback-full} if for any sequence $\seq{\bar{f}_i: B_i \rightarrow C \mid i \in I}$ of morphisms in $\cl$ and any sequence $\seq{g_i : A \rightarrow F B_i \mid i \in I}$ of morphisms in $\ck$, if $(F \bar{f}_i) g_i =  (F \bar{f}_j) g_j$ for all $i, j \in I$, then there exists an object $A'$ in $\cl$ and morphisms $\seq{\bar{g}_i: A' \rightarrow B_i \mid i \in I}$ such that $F \bar{g}_i = g_i$ for all $i \in I$.
\end{defin}

When is the functor $F$ from Theorem \ref{pres-thm} pullback-full? The answer yields yet another characterization of $\mu$-AECs admitting intersections:

\begin{thm}\label{pres-thm-char}
  Let $\K$ be a $\mu$-AEC. The following are equivalent:

  \begin{enumerate}
  \item\label{pres-thm-char-1} $\K$ admits intersections.
  \item\label{pres-thm-char-2} There is a universal $\mu$-AEC $\cl$ and an essentially surjective \emph{pullback-full} and faithful functor $F: \cl \rightarrow \K$ preserving $\mu$-directed colimits.
  \item\label{pres-thm-char-3} There is a $\mu$-AEC admitting intersections $\cl$ and an essentially surjective \emph{pullback-full} and faithful functor $F: \cl \rightarrow \K$ preserving $\mu$-directed colimits.
  \end{enumerate}
\end{thm}
\begin{proof} \
  \begin{itemize}
  \item \underline{(\ref{pres-thm-char-2}) implies (\ref{pres-thm-char-3})}: Trivial.
  \item \underline{(\ref{pres-thm-char-3}) implies (\ref{pres-thm-char-1})}:
    By Fact \ref{loc-limit-charact}, it suffices to check that $\K$ has wide pullbacks. Observe that the definition of pullback-fullness (taken with $I$ a singleton set and $\bar{f}_i = \id_{C}$) implies that whenever $g : A \rightarrow F C$, there exists $\bar{g} : A' \rightarrow C$ such that $F \bar{g} = g$. The proof is now routine using the fact that $\cl$ has wide pullbacks, together with the pullback-fullness and faithfulness of $F$.
  \item \underline{(\ref{pres-thm-char-1}) implies (\ref{pres-thm-char-2})}:
    In the proof of Lemma \ref{pres-univ}, let $\K''$ be the class of members $M'$ of $\K'$ such that $\ccl^{M'} (M_0) = M_0$ for all $M_0 \lea M' \rest \tau$. The reduct map from $\K''$ to $\K$ is a faithful functor from $\K''$ to $\K$ preserving $\mu$-directed colimits. As in the proof of \cite[2.11]{ap-universal-apal}, this functor is onto, and we show that it is also pullback-full.

    Let $A, C, \seq{B_i, \bar{f}_i, g_i \mid i \in I}$ be as in the definition of being pullback-full. For $i \in I$, let $h_i := (F \bar{f}_i) g_i$. By hypothesis, $h_i = h_j$ for all $i, j \in I$, so let $h := h_i$. First observe that $h [U A]$ induces a substructure $C_0$ of $C$, because $h[A] \lea C \rest \tau = F C$ and so by definition of $\K''$, $\ccl^{C} (h [U A]) = h[U A]$. We can find a $\tau'$-expansion $A'$ of $A$ such that $h$ induces a $\tau'$-isomorphism from $A'$ onto $C_0$. Thus $h$ will induce a $\K''$-embedding $\bar{h}$ from $A'$ into $C$.

    We now claim that for all $i \in I$, $g_i$ induces a $\K''$-embedding $\bar{g}_i$ from $A'$ into $B_i$. This is enough to prove what we want. To prove the claim, it is enough to see that for any $\alpha$-ary $\tau'$-function symbol $\rho$ and any $\ba \in \fct{\alpha}{A'}$, $g_i (\rho^{A'} (\ba)) = \rho^{B_i} (g_i (\ba))$. Now by the previous paragraph $h (\rho^{A'} (\ba)) = \rho^{C} (h (\ba))$. Since $C_0 = h[U A] = h[U A']$, $\rho^{C} (h (\ba)) = \rho^{C'} (h (\ba)) = \rho^{\bar{f}_i[B_i]} (h (\ba))$. By definition of $h$, the latter is equal to $\rho^{\bar{f}_i[B_i]} (\bar{f}_i g_i (\ba)) = \bar{f}_i (\rho^{B_i} (g_i (\ba))$. Putting these calculations together, we have obtained the equality $h (\rho^{A'} (\ba)) = \bar{f}_i (\rho^{B_i} (g_i (\ba)))$. Since (as a concrete function) $h = \bar{f}_i g_i$ and $\bar{f}_i$ is a monomorphism, we conclude that $g_i (\rho^{A'} (\ba)) = \rho^{B_i} (g_i (\ba))$, as desired.
  \end{itemize}
\end{proof}

Note that (\ref{pres-thm-char-3}) implies (\ref{pres-thm-char-1}) holds more generally when we start with an accessible category with $\mu$-directed colimits. We do not know whether (\ref{pres-thm-char-1}) implies (\ref{pres-thm-char-2}) also holds at that level of generality.

\bibliographystyle{amsalpha}
\bibliography{univ}

\end{document}